\documentclass[10pt]{article}
\font\smallit=cmti10

\usepackage{amssymb,latexsym,amsmath,epsfig,amsthm,xcolor} 
\usepackage{dirtytalk}
\usepackage{tikz}
\usepackage{cleveref}
\usepackage{url}
\makeatletter

\renewcommand\section{\@startsection {section}{1}{\z@}
{-30pt \@plus -1ex \@minus -.2ex}
{2.3ex \@plus.2ex}
{\normalfont\normalsize\bfseries\boldmath}}

\renewcommand\subsection{\@startsection{subsection}{2}{\z@}
{-3.25ex\@plus -1ex \@minus -.2ex}
{1.5ex \@plus .2ex}
{\normalfont\normalsize\bfseries\boldmath}}

\renewcommand{\@seccntformat}[1]{\csname the#1\endcsname. }

\makeatother

\newtheorem{theorem}{Theorem}
\newtheorem{lemma}{Lemma}
\newtheorem{proposition}{Proposition}
\newtheorem{corollary}{Corollary}

\theoremstyle{definition}
\newtheorem{definition}{Definition}

\newtheorem{example}{Example}

\begin{document}
\begin{figure}
\vspace*{-62pt}
\vspace{.75in}
\end{figure}

\begin{center}
\uppercase{\bf Fast winning strategies in a generalized van der Waerden game}
\vskip 20pt
{\bf Hannah Alpert\footnote{Supported by Simons Foundation grant No. 965348.}}\\
{\smallit Department of Mathematics and Statistics, Auburn University, Auburn, AL, USA}\\
{\tt hca0013@auburn.edu}\\ 
\vskip 10pt
{\bf Liam Barham}\\
{\smallit Department of Mathematics and Statistics, Auburn University, Auburn, AL, USA}\\
{\tt blb0081@auburn.edu}\\ 
\vskip 10pt
{\bf Brian Freidin}\\
{\smallit Department of Mathematics and Statistics, Auburn University, Auburn, AL, USA}\\
{\tt bgf0012@auburn.edu}\\ 
\vskip 10pt
{\bf Ian Tan}\\
{\smallit Department of Mathematics and Statistics, Auburn University, Auburn, AL, USA}\\
{\tt yzt0060@auburn.edu}\\ 
\vskip 10pt
{\bf Alexandra Weiner}\\
{\smallit Department of Mathematics and Statistics, Auburn University, Auburn, AL, USA}\\
{\tt gaw49@pitt.edu}\\ 
\end{center}
\vskip 20pt

\centerline{\bf Abstract}
\noindent
Consider the following Maker-Breaker game. Fix a finite subset $S\subset\mathbb{N}$ of the naturals. The players Maker and Breaker take turns choosing previously unclaimed natural numbers. Maker wins by eventually building a homothetic copy $aS+b$ of $S$, where $a\in\mathbb{N}\setminus\{0\}$ and $b\in\mathbb{Z}$. This is a generalization of the van der Waerden game analyzed by Beck. By the Hales-Jewett theorem, there exists a constant $c$ depending only on $|S|$ such that Maker can win in $c$ or less moves. We show that Maker can win in $|S|$ moves if $|S|\leq 3$. When $|S|=4$, we show that Maker can always win in $5$ or less moves and describe all $S$ such that Maker can win in $4$ moves. If $|S|\geq 5$, Maker has no winning strategy in $|S|$ moves.

\pagestyle{myheadings}
\markright{Fast winning strategies in a generalized van der Waerden game\hfill}
\thispagestyle{empty}
\baselineskip=12.875pt
\vskip 30pt

\section{Introduction}

Let $S\subset \mathbb{N}:=\{0,1,2,\dots\}$ be a finite set and consider the following game between players Maker and Breaker. The players alternately pick previously unselected natural numbers, with Maker making the first selection. Maker wins when for some $a\in\mathbb{N}\setminus\{0\}$ and $b\in\mathbb{Z}$, the set of Maker's choices contains $aS+b$.

This is a generalization of the \textit{van der Waerden game} analyzed by Beck in \cite{Beck80}. In the van der Waerden game, players pick elements of the finite set $[m]:=\{1,2,\dots,m\}$ and Maker wins by selecting all elements of an $n$-term arithmetic progression. Beck's interest is in the minimum $m$ such that Maker has a winning strategy.
The van der Waerden game is related to the following theorem.

\begin{theorem}[Van der Waerden's theorem, \cite{Waerden27}]\label{vdwthm}
    For all $n,r$ there exists $m$ such that every $r$-coloring $\Delta:[m]\to [r]$ of $[m]$ contains a monochromatic arithmetic progression $S$ of size $n$. That is, $S=\{s,s+a,s+2a,\dots,s+(n-1)a\}\subset [m]$ such that $|\Delta(S)|=1$.
\end{theorem}

A longstanding open problem is to determine, for any $n,r$, the minimum $m$ (called a van der Waerden number) that satisfies the statement of Theorem \ref{vdwthm} \cite{GRS90}. 
Brown et al. \cite{BLM97} consider a more general problem in the case that $n=3$ and $r=2$: for every $S\subset \mathbb{N}$ of size 3, they ask for the minimum $m$ such that any 2-coloring of $[m]$ contains $aS+b$ for some $a\in\mathbb{N}\setminus\{0\}$ and $b\in\mathbb{Z}$. This was completely solved in subsequent work by Kim and Rho \cite{KR12}. The game introduced in this paper generalizes the van der Waerden game in a similar manner.

By Theorem \ref{vdwthm}, Maker can always win our generalized van der Waerden game. Indeed, for any $S\subset\mathbb{N}$ we can pick $n'\in\mathbb{N}$ such that $S\subset [n']$. Theorem \ref{vdwthm} implies that there exists $m$ such that any 2-coloring of $[m]$ contains a monochromatic $n'$-term arithmetic progression $S'$. Furthermore, this arithmetic progression contains $aS+b\subset S'$ for some $a\in\mathbb{N}\setminus\{0\}$ and $b\in\mathbb{Z}$. Since Maker has the advantage of making the first move, Maker can claim the monochromatic $aS+b$ and win \cite{Beck19}. In fact, this argument shows that Maker can always win by selecting naturals from $[m]$. Thus, we could play our game on the finite board $[m]$ instead of $\mathbb{N}$. In this case, an interesting and natural problem to consider is finding the minimum $m$ such that Maker has a winning strategy. However, our focus is on minimizing the number of moves needed for Maker to win. So we choose the infinite board $\mathbb{N}$ for convenience.

\section{Literature Review}

More generally, any hypergraph $(V,\mathcal{F})$ consisting of a set $V$ and a family $\mathcal{F}\subset\mathcal{P}(V)$ of subsets of $V$ corresponds to a Maker-Breaker game. The set $V$ is called the \textit{board} and the elements of $\mathcal{F}$ are called \textit{winning sets}. Two players, Maker and Breaker, alternately pick previously unselected elements of $V$, with Maker making the first selection. Maker wins by claiming all elements of a winning set, while Breaker wins by preventing Maker from winning.

In many families of Maker-Breaker games, Maker has a winning strategy if the board is large enough (such as in the van der Waerden game). This leads to the concept of ``biased'' games, in which Maker selects one element on each turn and Breaker selects $q$ elements on each turn. Define the \textit{threshold bias} to be the integer $q'$ such that the game is Maker's win if and only if $q<q'$. Over the years, much effort has been put into understanding the threshold bias for various Maker-Breaker games (for example, see \cite{Bednarska00,Chvatal78,Glazik22,Hancock19,Krivelevich10}). Another approach, which we take in this paper, is to find fast winning strategies for Maker. In particular, there exists a minimal $r$ such that Maker can always win in $r$ or less moves; there has also been considerable work on understanding the value of $r$ for various games (for example, see \cite{CFKL12, Gledel19, HKSS07, Mikalacki18}).

For a reference on Maker-Breaker games and, more generally, positional games, we recommend the book \cite{Hefetz14}.

\section{Overview of the Paper}
Hereafter, the only game we we will consider is the one described at the beginning of this paper. In Section~\ref{sec:upper}, we provide an upper bound on the number of moves needed for Maker to win the game, depending only on $|S|$. In Section \ref{sec:winning}, we explain a method for finding winning strategies by constructing trees with certain properties, concluding with a construction that yields a win for Maker in $|S|$ moves if $|S|\leq 3$. Then in Section~\ref{sec:weak}, we introduce the notion of ``symmetric'' sets, which we use to show that Maker has no winning strategy in $|S|$ moves if $|S|\geq 5$. Section~\ref{sec:four} solves the case where $|S|=4$; we show that Maker can win in 5 moves or less, and characterize $S$ such that Maker can win in 4 moves (these turn out to be the symmetric sets first defined in Section~\ref{sec:weak}).

\section{Upper Bound for Sets of Fixed Size}\label{sec:upper}
As mentioned in the introduction, Maker wins the game for any $S\subset \mathbb{N}$, by Theorem~\ref{vdwthm}.  However, a more nuanced argument shows that there exists an upper bound on the number of moves that Maker needs to win, depending only on $|S|=n$. We thank Dömötör Pálvölgyi for bringing to our attention the proof of this fact, which is an application of the Hales-Jewett theorem. While the argument can be found in \cite{GRS90}, we provide the details here for ease to the reader.

\begin{definition}
    A \textit{line} in
    $
    S^N=\{s_1,\dots,s_n\}^N
    $
    is a collection of $n$ distinct points $\textbf{x}^1,\dots,\textbf{x}^n\in S^N$ (indexed appropriately) such that, for each $1\leq j\leq N$, either
    \begin{enumerate}
        \item $\textbf{x}^1_j=\textbf{x}^2_j=\dots=\textbf{x}^n_j$, or
        \item $\textbf{x}^k_j=s_k$ for all $1\leq k\leq n$.
    \end{enumerate}
\end{definition}

\begin{theorem}[The Hales-Jewett Theorem, \cite{GRS90}]
    For all $r,n$ there exists $N=HJ(r,n)$ such that every $r$-coloring of $S^N$ contains a monochromatic line.
\end{theorem}

\begin{theorem}\label{HJ}
    For any $S$ of size $n$, Maker has a strategy that wins in at most $\lceil \frac{1}{2}n^{HJ(2,n)}\rceil$ moves.
\end{theorem}
\begin{proof}
    Choose $m$ such that $S\subset \{0,1,\dots,m-1\}$ and set $N=HJ(2,n)$. Then the map
    \[
    \Psi:S^N\to\mathbb{N},\quad (x_0,\dots,x_{N-1})\mapsto \sum_{i=0}^{N-1} x_im^i
    \]
    is injective. Maker can adopt the strategy of selecting numbers from the range of $\Psi$. By injectivity, this produces a 2-coloring of $S^N$. By the Hales-Jewett theorem, there is a monochromatic line in $S^N$, which corresponds to a monochromatic copy of $S$. The monochromatic copy can be claimed by the player with the first move: Maker \cite{Beck19}.
\end{proof}

Let $C_n$ be the maximum number of moves that Maker needs to win for any $S$ of size $n$. By Theorem~\ref{HJ}, we have $C_n\leq\lceil\frac{1}{2}n^{HJ(2,n)}\rceil$. However, the function $HJ(r,n)$ belongs to the class $\mathcal{E}^5$ of the Grzegorczyk hierarchy \cite{Shelah88} and thus grows extremely fast. We leave as an open question the value of $C_n$ for arbitrary $n$; it might be more approachable to only improve the upper bound. By the results of this paper (Corollary \ref{cor:leq3} and Theorem \ref{5}), we have $C_n=n$ when $n\leq 3$ and $C_4=5$.

\section{Winning Strategies from Trees}\label{sec:winning}
We say that Maker has an \textit{$p$-move strategy} if there exists a strategy such that Maker always wins with $p$ selections or less. Although we initially assumed that $S\subset \mathbb{N}$, it will be convenient to talk about subsets $S\subset\mathbb{Q}$. Given $S,S'\subset\mathbb{Q}$, we say that $S'$ is a \textit{copy} of $S$ if there exists $a,b\in\mathbb{Q}$ with $a>0$ such that $S'=aS+b$. In other words, there exists an increasing affine function $f:\mathbb{Q}\to\mathbb{Q}$ such that $S'=f(R)$.
\begin{proposition}\label{prop:tree}
    Suppose that there exists a directed tree $T=(V,E)$ of height $p-1$ with the following properties:
    \begin{enumerate}
        \item The vertices are distinct rational numbers.
        \item The set of rationals along any maximal directed path contains a copy of $S$.
        \item Every vertex which is neither a leaf nor the root has outdegree at least 2.
    \end{enumerate}
    Then Maker has an $p$-move strategy.
\end{proposition}
\begin{proof}
Suppose such a tree $T$ exists. We may assume that the root of $T$ is $0$ (if not, translate by the appropriate constant). Choose $k\in\mathbb{Q}$ large enough so that 
$
    kV\cap V=\{0\}.
$
Let $T'$ be the isomorphic tree obtained by replacing the vertex set of $T$ with $kV$. We then construct a larger tree $T''$ by taking the union of $T$ and $T'$ then identifying their roots. The tree $T''$ still satisfies the three properties listed above, but we also have that the root of $T''$ has outdegree 2. The first property is satisfied by choice of $k$. The second and third properties are satisfied because they hold for $T$ and $T'$. We can further assume that the vertices of $T''$ are distinct naturals, since there exists some increasing affine function $f:\mathbb{Q}\to\mathbb{Q}$ such that $f(V\cup kV)\subset\mathbb{N}$. The function $f$ is the composition of two operations: given any finite set of rationals we can multiply by the least common multiple of the denominators to obtain a set of integers, then add a suitable integer to obtain a set of naturals.

Now Maker can adopt the following strategy. Maker's first selection is the root of $T''$, and each subsequent selection is an outneighbor of the previous one such that Breaker has not selected any vertices on the subtree starting from Maker's selection. Since all the vertices are distinct and all non-leaf vertices have outdegree at least 2, on each move Breaker can pick an element from at most one subtree starting from an outneighbor of Maker's previous move. By following this strategy, Maker selects integers from a directed path. Maker wins in at most $p$ moves since every maximal directed path contains a copy of $S$ and the height of the tree is $p-1$.
\end{proof}

From now on, we will look for trees as described in the proposition, as they correspond to winning strategies for Maker.
\begin{example}\label{1234}\normalfont
    Figure~\ref{fig:1234} displays a tree giving a $4$-move strategy for Maker when $S=\{1,2,3,4\}$. Intuitively, Proposition~\ref{prop:tree} states that the game's board may as well be $\mathbb{Q}$ and that Maker may as well select two elements on the first move. From this perspective, Maker wins by first selecting $0$ and $1$, then making successive selections by following a directed path.
\end{example}
\begin{example}\label{0236}
    Figure~\ref{fig:1347} displays a tree giving a 5-move strategy when $S=\{0,2,3,6\}$.
\end{example}

\begin{figure}
    \centering
     \begin{tikzpicture}
  \node(0) at (0,1) {$0$};
  \node (1) at (0,0) {$1$};
  \node (c) at (3,-1.5) {$2$};
  \node (d) at (-3,-1.5) {$\frac{1}{2}$};
  \node (a) at (-4.5,-3) {$-\frac{1}{2}$};
  \node (b) at (-1.5,-3) {$\frac{3}{2}$};
  \node (f) at (1.5,-3) {$-1$};
  \node (g) at (4.5,-3) {$3$};
  \draw[->] (0) edge (1) (1) edge (c) (1) edge (d) (c) edge (f) (c) edge (g) (d) edge (a) (d) edge (b);
\end{tikzpicture}
    \caption{A tree giving a winning strategy for $\{1,2,3,4\}$.}
    \label{fig:1234}
\end{figure}

\begin{figure}
    \centering
     \begin{tikzpicture}
  \node (1) at (0,1) {$0$};
  \node (2) at (0,0) {$48$};
  \node (h) at (-3,-1.5) {$16$};
  \node (j) at (3,-1.5) {$12$};
  \node (k) at (-5,-1.5) {$24$};
  \node (l) at (5,-1.5) {$-24$};
  \node (c) at (3,-3) {$-6$};
  \node (d) at (-3,-3) {$64$};
  \node (a) at (-4.5,-4.5) {$-32$};
  \node (b) at (-1.5,-4.5) {$112$};
  \node (f) at (1.5,-4.5) {$21$};
  \node (g) at (4.5,-4.5) {$3$};
  \draw[->] (1) edge (2);
  \draw[->] (2) edge (h) (h) edge (d)  (2) edge (j) (j) edge (c) (j) edge (l) (h) edge (k) (c) edge (f) (c) edge (g) (d) edge (a) (d) edge (b); 
\end{tikzpicture}
    \caption{A tree giving a winning strategy for $\{0,2,3,6\}$.}
    \label{fig:1347}
\end{figure}

\begin{figure}
    \centering
    \begin{tikzpicture}
        \node (0) at (0,1) {$0$};
        \node (1) at (0,0) {$1$};
        \node (2) at (2,-1) {$\frac{b+c}{b}$};
        \node (3) at (-2,-1) {$\frac{b}{b+c}$};        
        \draw[->] (0) edge (1) (1) edge (2) (1) edge (3);
    \end{tikzpicture}
    \caption{A tree giving a winning strategy for sets of size 3.}
    \label{fig:win3}
\end{figure}

\begin{corollary}\label{cor:leq3}
    When $|S|\leq 3$, Maker has a $|S|$-move strategy.
\end{corollary}
\begin{proof}
    If $|S|$ is $1$ or $2$ the statement is trivial. Suppose $|S|=3$ and write $S=\{a,a+b,a+b+c\}$. Consider the tree in Figure \ref{fig:win3}. Proposition~\ref{prop:tree} completes the proof.
\end{proof}

\section{A Weak Lower Bound}\label{sec:weak}
From now on, we use the convention that $S=\{s_1,\dots,s_n\}$ with its elements indexed in increasing order (in particular, $|S|=n$).
The best hope for Maker is to win in $n$ moves. A necessary condition on $S$ for this to be guaranteed is that after $n-1$ moves, there are at least two different ways to complete a copy of $S$. Otherwise, Breaker would take the only move. This motivates the following definition.
\begin{definition}
\normalfont
    The set $S$ is called \textit{symmetric} if there exists $1\leq i<j\leq n$ such that $S\setminus\{s_i\}$ is a copy of $S\setminus\{s_j\}$. In this case we say that $S$ has the \textit{type} $(i,j)$.
\end{definition}
\begin{example}\normalfont
    The arithmetic sequence $\{1,2,\dots,n\}$ is symmetric of type $(1,n)$. The set $\{1,2,3,5\}$ is symmetric of type $(2,4)$ since $\{1,3,5\}=2\{1,2,3\}-1$.
\end{example}
\begin{lemma}\label{nec}
    Maker can only have an $n$-move strategy if $S$ is symmetric. 
\end{lemma} 
\begin{proof}
    By Corollary \ref{cor:leq3}, if $n\leq 3$ then Maker has an $n$-move strategy; but in this case $S$ is symmetric of every possible type. Now, assume that $n\geq 4$. Suppose Maker has an $n$-move strategy. Using this strategy, Maker must construct $f(S\setminus\{s_i\})$ after $n-1$ moves for some increasing affine $f:\mathbb{Q}\to\mathbb{Q}$. Breaker can then select $f(s_i)$ in an attempt to prevent Maker's win. The fact that Maker can still win means that there exists $s_j\in S$ and an increasing affine $g:\mathbb{Q}\to\mathbb{Q}$ such that $f(S\setminus \{s_i\})=g(S\setminus\{s_j\})$. We have $i\neq j$, for otherwise the increasing affine function $g^{-1}\circ f$ fixes $S\setminus\{s_i\}$ pointwise, which is impossible since $f\neq g$.
\end{proof}
\begin{lemma}\label{types}
        A symmetric set of size $n\geq 4$ can only be of type $(1,n)$, $(2,n)$ or $(1,n-1)$.
\end{lemma}
\begin{proof}
Suppose that $S$ is symmetric of type $(i,j)$. If $i>2$ (resp. $j<n-1$), then the first (resp. last) two elements of $S$ are also the first (resp. last) two elements of $S\setminus\{s_i\}$ and $S\setminus\{s_j\}$. This is a contradiction since an affine function $f:\mathbb{Q}\to\mathbb{Q}$ does not fix two distinct points unless $f$ is the identity function. The only remaining case to rule out is $(i,j)=(2,n-1)$. However, in this case the first and last elements are the same in $S\setminus\{s_i\}$ and $S\setminus\{s_j\}$.
\end{proof}

The group of increasing affine functions $\mathbb{Q}\to\mathbb{Q}$ acts on the set of symmetric subsets $S\subset\mathbb{Q}$. The type of $S$ is an invariant of this group action. Our next goal is to show that symmetric sets that are not arithmetic sequences can be characterized by their size, type, and a rational $k>1$. We do this by giving explicit representatives in each orbit of the group action.

\begin{definition}\normalfont
    For $n$ and $k\in \mathbb{Q}$ such that $k>1$, define
    \[
        S_1(n,k)=\{1,k,k^2,\dots,k^{n-1}\}
    \]\
    and
\begin{equation*}
    S_2(n,k)=
    \begin{cases}
        \{0\} & \text{if $n=1$,} \\
        \{0,1,k,k^2,\dots,k^{n-2}\} & \text{if $n\geq 2$.}
    \end{cases}
\end{equation*}
\end{definition}

\begin{lemma}\label{canon}
    Suppose $S$ is symmetric of type $(1,n)$ and is not an arithmetic sequence. Then there exists $k>1$ such that $S_1(n,k)$ is a copy of $S$.
\end{lemma}
\begin{proof}
By assumption, we have that $S\setminus\{s_1\}=f(S\setminus\{s_n\})$ for some increasing affine function $f(x)=kx+b$, with $k,b\in\mathbb{Q}$ and $k>0$. Expanding out this equation,
\[
\{s_2,s_3,\dots,s_n\}=\{f(s_1),f(s_2),\dots,f(s_{n-1})\}.
\]
Since increasing functions preserve order, this means that
\begin{equation}\label{eq:iterate}
s_{i+1}=f(s_i),\quad \text{for all $i\in\{1,\dots,n-1\}$.}
\end{equation}
We may assume that $s_1=0$ and $s_2=1$ (if not, apply an appropriate increasing affine function). Then $f(0)=1$, which implies that $b=1$. Furthermore, by Equation \eqref{eq:iterate}, we have
\[
S=\{0,1,1+k,1+k+k^2,\dots\} \quad\text{and}\quad
f(S)=\{1,1+k,1+k+k^2,\dots\}.
\]
Since $S$ is not an arithmetic sequence, $k\neq 1$. Using the geometric series formula $\sum_{i=0}^{n-1}k^i=\frac{k^n-1}{k-1}$ we have
$$(k-1)\cdot f(S)+1=\{k,k^2,\dots,k^{n}\}.$$
If $k>1$, note that $S_1(n,k)$ is a copy of the above. If $k<1$, note that $S_1(n,\frac{1}{k})$ is a copy of the above.
\end{proof}

Given $S,S'\subset\mathbb{Q}$, we say that $S'$ is a \textit{reflection} of $S$ if there exists $a,b\in\mathbb{Q}$ with $a<0$ such that $S'=aS+b$. This definition is similar to that of ``copy,'' but with $a$ being negative instead of positive. 
of $S$. If $S$ has symmetry type $(2,n)$, then a reflection of $S$ has symmetry type $(1,n-1)$.

\begin{lemma}\label{canon2}
    Suppose $S$ is symmetric of type $(2,n)$ (resp. $(1,n-1)$). Then there exists some $k>1$ such that $S_2(n,k)$ is a copy (resp. reflection) of $S$.
\end{lemma}
\begin{proof}
    Suppose $S$ is symmetric of type $(2,n)$. Taking an approach similar to the proof of Lemma \ref{canon}, we have
    \[
    \{s_1,s_3,s_4,\dots,s_n\}=\{f(s_1),f(s_2),f(s_3),\dots,f(s_{n-1})\}
    \]
    for some increasing affine function $f(x)=kx+b$, with $k,b\in\mathbb{Q}$ and $k>0$. Then
    \begin{equation*}
    f(s_i)=
        \begin{cases}
            s_i & \text{if $i=1$,} \\
            s_{i+1} & \text{if $i>1$.}
        \end{cases}
    \end{equation*}
    We may assume that $s_1=0$ and $s_2=1$. Then $f(0)=0$, which implies that $b=0$. It follows that $S=S_2(n,k)$, where $k>1$ by the assumption that $s_1<\dots<s_n$. If $S$ is symmetric of type $(1,n-1)$, then $-S$ is symmetric of type $(2,n)$ and therefore a copy of $S_2(n,k)$. This means that $S$ is a reflection of $S_2(n,k)$.
\end{proof}
Lemmas \ref{canon} and \ref{canon2} reveal a peculiar property of symmetric sets. If $S$ is symmetric of type $(i,j)$, then $S\setminus\{s_i\}$ and $S\setminus\{s_j\}$ are also symmetric of a similar type. This can be checked directly from the definitions of $S_1(n,k)$ and $S_2(n,k)$. For example, a symmetric set $S$ with symmetry type $(1,n)$ is a copy of $S_1(n,k)=\{1,k,\dots,k^{n-1}\}$. Then $S\setminus\{s_n\}$ is a copy of $\{1,k,\dots,k^{n-1}\}\setminus\{k^{n-1}\}=S_1(n-1,k)$ and, by the definition of a symmetric set, so is $S\setminus\{s_1\}$.
\begin{lemma}\label{nos}
    If $n\geq 4$ then $S$ cannot be simultaneously symmetric of two different types.
\end{lemma}
\begin{proof}
    By Lemma \ref{types}, we can prove the claim by checking three cases.
    Suppose $S$ is simultaneously symmetric of types $(1,n)$ and $(2,n)$. Then $S\setminus\{s_1\}$ is a copy of $S\setminus\{s_n\}$ and $S\setminus\{s_2\}$ is a copy of $S\setminus\{s_n\}$. It follows that $S\setminus\{s_1\}$ is a copy of $S\setminus\{s_2\}$, i.e. $S$ also has symmetry type $(1,2)$. But this is impossible, by Lemma \ref{types}. Similarly, $S$ cannot be simultaneously symmetric of types $(1,n)$ and $(1,n-1)$. This leaves one remaining case. \par

    Suppose $S$ is simultaneously symmetric of types $(2,n)$ and $(1,n-1)$. By Lemma \ref{canon2}, there exist increasing affine functions $f,g:\mathbb{Q}\to\mathbb{Q}$ and constants $k,l>1$ such that
    $$f(\{0,1,k,\dots,k^{n-3},k^{n-2}\})=S=g(\{-l^{n-2},-l^{n-3},\dots,-l,-1,0\}).$$
    Let $h=g^{-1}\circ f$. Since increasing functions preserve order, $h(0)=-l^{n-2}$, $h(1)=-l^{n-3}$, $h(k^{n-3})=-1$ and $h(k^{n-2})=0$. If $h(x)=ax+b$, the first equation implies $b=-l^{n-2}$. From the other three we get the following system.
    \begin{align*}
        a-l^{n-2}&=-l^{n-3}\\
        ak^{n-3}-l^{n-2}&=-1\\
        ak^{n-2}-l^{n-2}&=0
    \end{align*}
    Solving for $a$ and $k$ in terms of $l$, we obtain $a=l^{n-2}-l^{n-3}$ and $k=\frac{l^{n-2}}{l^{n-2}-1}$. Then we plug these values into $ak^{n-2}-l^{n-2}=0$ and clear the denominator to get
    $$(l^{n-2}-l^{n-3})l^{(n-2)^2}-l^{n-2}(l^{n-2}-1)^{n-2}=0.$$
    By the rational root theorem, the only possible rational solutions are $l=\pm 1$. But this contradicts $l>1$.
\end{proof}
We find it interesting to note that if $S$ is allowed to contain any real number, $S=\{0,1,\varphi,\varphi^2\}$ is simultaneously symmetric of types $(2,4)$ and $(1,3)$, where $\varphi=\frac{1+\sqrt{5}}{2}$ is the golden ratio.

We end this section with one of the main claims of this paper (Theorem \ref{gab}). Its proof uses the following lemma. 
\begin{lemma}\label{han}
    Suppose Maker has an $n$-move strategy. In any game using this winning strategy, on the $(n-r)$th move Maker builds either
    \begin{itemize}
        \item a copy of $S_1(n-r,k)$, if $S$ has type $(1,n)$,
        \item a copy of $S_2(n-r,k)$, if $S$ has type $(2,n)$, or
        \item a reflection of $S_2(n-r,k)$, if $S$ has type $(1,n-1)$.
    \end{itemize}
    Hence, Maker has an $(n-r)$-move strategy for symmetric sets of size $n-r$ of the corresponding type.
\end{lemma}
\begin{proof}
Note the assumption that $S$ is symmetric, which holds by Lemma \ref{nec}. Suppose $S$ has type $(i,j)$. It suffices to prove the case $r=1$. The result follows by induction. Suppose Maker builds $R$ on their $(n-1)$th move; we may assume that $|R|> 3$ since sets of size 3 are symmetric of every possible type. We may assume that Breaker would pick $x$ such that $R\cup\{x\}$ is a copy of $S$. Maker then plays $y\neq x$ such that $R\cup\{y\}$ is also a copy of $S$. Thus, $R\cup\{y\}$ is a copy of $R\cup\{x\}$. That is, there exists an increasing affine function $f:\mathbb{Q}\to\mathbb{Q}$ such that
$$f(R)\cup\{f(x)\}=R\cup\{y\}=S'$$
for some copy $S'=\{s'_1,s'_2,\dots,s'_n\}$ of $S$, indexed in increasing order. Note that $f(x)\neq y$, otherwise $f(R)=R$ which is impossible if $|R|>1$. It follows that $S'\setminus\{y\}$ is a distinct copy of $S'\setminus\{f(x)\}$. This means that $S'$ is symmetric of type $(k,l)$ where $f(x)=s'_k$ and $y=s'_l$ (or vice versa). By Lemma \ref{nos}, $(k,l)=(i,j)$. Therefore, $R=S'\setminus\{s'_j\}$. By the comments preceding Lemma \ref{nos}, we are done. 
\end{proof}
\begin{theorem}\label{gab}
    If $n\geq 5$, then Maker does not have an $n$-move strategy.
\end{theorem}
\begin{proof}
We only need to prove the case $n=5$; if Maker has an $n$-move strategy and $n>5$, then Maker would have a 5-move strategy by the last sentence of Lemma \ref{han}. Suppose for the sake of a contradiction that Maker has a $5$-move strategy. We break the argument into cases by Lemma \ref{types}.

Suppose $S$ has symmetry type $(1,5)$ and is not an arithmetic sequence. Again by Lemma \ref{han}, Maker builds $f(\{k,k^2,k^3\})$ on their third move and must build $g(\{1,k,k^2,k^3\})$ on their fourth move, where $f,g:\mathbb{Q}\to\mathbb{Q}$ are increasing affine functions and $k>1$. We claim that the only way to do this is for Maker to select $f(1)$ or $f(k^4)$ on their fourth move. Let $x\in\mathbb{Q}$ such that $$f(\{k,k^2,k^3\}\cup\{x\})=g(\{1,k,k^2,k^3\}).$$ If $x<k$ then $x=1$, otherwise $g^{-1}\circ f$ is not the identity but fixes $k,k^2,$ and $k^3$. Similarly, if $x>k^3$ then $x=k^4$. To exclude the remaining possibilities, note:
\begin{itemize}
    \item if $k<x<k^2$ then $\{1,k,x,k^2,k^3\}$ is symmetric of type $(1,3)$, and
    \item if $k^2<x<k^3$ then $\{k,k^2,x,k^3,k^4\}$ is symmetric of type $(3,5)$.
\end{itemize}
Neither of the above is possible (Lemma \ref{types}), which proves the claim.
    Now, if Breaker selects $f(1)$ on their third move, then Maker is forced to select $f(k^4)$ on their fourth move.
    
    Maker has now built $f(\{k,k^2,k^3,k^4\})$ and must on their fifth move build a copy of $S$. More explicitly, Maker must find $y\in\mathbb Q$ such that for some increasing affine function $h:\mathbb Q\to\mathbb Q$, $$f(\{k,k^2,k^3,k^4\}\cup\{y\})=h(\{1,k,k^2,k^3,k^4\}).$$ Similar to the argument above, we claim that either $y=1$ or $y=k^5$ by the following: \begin{itemize}
        \item if $k<y<k^2$, then $\{1,k,y,k^2,k^3,k^4\}$ is symmetric of type $(1,3)$;
        \item if $k^2<y<k^3$, then $\{1,k,k^2,y,k^3,k^4\}$ is symmetric of type $(1,4)$;
        \item if $k^3<y<k^4$, then $\{k,k^2,k^3,y,k^4,k^5\}$ is symmetric of type $(4,6)$.
    \end{itemize} 
    None of the above are possible. Thus, the only way for Maker to build a copy of $\{1,k,\dots,k^4\}$ is to select $f(1)$ or $f(k^5)$. But $f(1)$ has already been selected and Breaker can select $f(k^5)$, stopping the win. If $S$ is an arithmetic sequence, replace $k^i$ with $i$. The arguments for the remaining cases are similar to the above and we state them with less detail.

    Suppose $S$ has symmetry type $(2,5)$. By Lemma \ref{han}, Maker builds $f(\{0,1,k\})$ on on their third move and must build a copy of $\{0,1,k,k^2\}$ on their fourth move, where $f:\mathbb{Q}\to\mathbb{Q}$ is an increasing affine function. The only way for Maker to do this is for Maker to select $f(\frac{1}{k})$ or $f(k^2)$ on their fourth move. If Breaker selects $f(\frac{1}{k})$ on their fourth move, then Maker is forced to select $f(k^3)$ on their fourth move. Now Breaker can select $f(k^4)$ and stop the win. If $S$ has symmetry type $(1,4)$, a similar argument holds with $f$ being a decreasing affine function.
\end{proof}

\section{Sets of Size Four}\label{sec:four}
In the case where $n=4$, we can find fast winning strategies for any $S$. We achieve this by explicit construction of trees that satisfy the properties of Proposition \ref{prop:tree}. We have the following results.
\begin{theorem}\label{4}
    Maker has a $4$-move strategy if and only if $S$ is symmetric.
\end{theorem}

\begin{theorem}\label{5}
    Maker always has a $5$-move strategy. 
\end{theorem}
\begin{proof}[Proof of Theorem \ref{4}]
    One direction follows from Lemma \ref{nec}. For the other direction, let $S$ be symmetric. If $S$ has type $(1,4)$, by (the proof of) Lemma \ref{canon}, we can take $\{0,1,1+k,1+k+k^2\}$ for some $k> 1$ to be a copy of $S$. Figure~\ref{fig:14} displays a tree giving a winning strategy.
The vertices are distinct since
$$-\frac{1}{k+k^2}<-\frac{1}{k}<0<\frac{1}{1+k}<1<\frac{1+k+k^2}{1+k}<1+k<1+k+k^2.$$ Suppose $S$ has type $(2,4)$.
Figure~\ref{fig:24} displays a tree giving a winning strategy.
The vertices are distinct since
$$-\frac{1}{k-1}<-\frac{1}{k}<0<\frac{1}{k^2}<\frac{1}{k}<1<k<1+k.$$
If $S$ has type $(1,3)$, we may take the tree of Figure~\ref{fig:24} and multiply through by $-1$.
\end{proof}

\begin{figure}
    \centering
     \begin{tikzpicture}[scale=0.9]
  \node(0) at (0,1) {$0$};
  \node (1) at (0,0) {$1$};
  \node (c) at (3,-1.5) {$1+k$};
  \node (d) at (-3,-1.5) {$\frac{1}{1+k}$};
  \node (a) at (-1.5,-3) {$\frac{1+k+k^2}{1+k}$};
  \node (b) at (-4.5,-3) {$-\frac{1}{k+k^2}$};
  \node (f) at (4.5,-3) {$1+k+k^2$};
  \node (g) at (1.5,-3) {$-\frac{1}{k}$};
  \draw[->] (0) edge (1) (1) edge (c) (1) edge (d) (c) edge (f) (c) edge (g) (d) edge (a) (d) edge (b);
\end{tikzpicture}
    \caption{A tree giving a winning strategy for symmetric sets of type $(1,4)$.}
    \label{fig:14}
\vspace{9.9mm}
    \centering
     \begin{tikzpicture}[scale=0.9]
  \node (0) at (0,1) {$0$};
  \node (1) at (0,0) {$1$};
  \node (c) at (3,-1.5) {$-\frac{1}{k-1}$};
  \node (d) at (-3,-1.5) {$\frac{1}{k}$};
  \node (a) at (-4.5,-3) {$\frac{1}{k^2}$};
  \node (b) at (-1.5,-3) {$k$};
  \node (f) at (4.5,-3) {$1+k$};
  \node (g) at (1.5,-3) {$-\frac{1}{k}$};
  \draw[->] (0) edge (1) (1) edge (c) (1) edge (d) (c) edge (f) (c) edge (g) (d) edge (a) (d) edge (b);
\end{tikzpicture}
    \caption{A tree giving a winning strategy for symmetric sets of type $(2,4)$ or $(1,3)$.}
    \label{fig:24}
\vspace{9.9mm}
    \centering
\begin{tikzpicture}[scale=0.9]
  \node (0) at (0,1) {$0$};
  \node (1) at (0,0) {$1$};
  \node (h) at (-3,-1.5) {$f_0$};
  \node (j) at (3,-1.5) {$f_1$};
  \node (k) at (-5,-1.5) {$f_{00}$};
  \node (l) at (5,-1.5) {$f_{10}$};
  \node (c) at (3,-3) {$f_{11}$};
  \node (d) at (-3,-3) {$f_{01}$};
  \node (a) at (-4.5,-4.5) {$f_{010}$};
  \node (b) at (-1.5,-4.5) {$f_{011}$};
  \node (f) at (1.5,-4.5) {$f_{110}$};
  \node (g) at (4.5,-4.5) {$f_{111}$};
  \draw[->] (0) edge (1) (1) edge (h) (h) edge (d)  (1) edge (j) (j) edge (c) (j) edge (l) (h) edge (k) (c) edge (f) (c) edge (g) (d) edge (a) (d) edge (b); 
\end{tikzpicture}
    \caption{A tree giving a winning strategy for sets of size four. The vertices are rational functions defined in Table~\ref{tab:polys}.}
    \label{fig:size4}
\end{figure}

\begin{proof}[Proof of Theorem \ref{5}]
Let $S=\{0,x,x+y,x+y+z\}$. Consider the tree $T$ from Figure~\ref{fig:size4}, whose vertices are defined in Table~\ref{tab:polys}.
Note that
\begin{align*}
    S &= \frac{xz}{x+y+z}\{0,f_0,f_{01},f_{010}\}+x+y\\
    &=\frac{xz}{y+z}\{1,f_0,f_{01},f_{011}\}-\frac{xz}{y+z}+x+y\\
    &=\frac{xy+y^2+xz+yz}{x+y+z}\{0,f_1,f_{11},f_{110}\}+x\\
    &=\frac{xy+y^2+xz+yz}{z}\{1,f_1,f_{11},f_{111}\}-\frac{xy+y^2+yz}{z}.
\end{align*}
If all of the vertices of $T$ are distinct, then Maker has a 5-move strategy. Unfortunately, this is not always the case. However, we will show that there still exist 5-move strategies for those cases missed by the tree.

Consider each vertex of $T$ to be a rational function in the differences $x,y,z$ defining $S$. Let $\mathbb{F}$ be the field of rational functions in $x,y,z$ and let $\Phi:\mathbb{F}\to\mathbb{F}$ be the automorphism that swaps the variables $x$ and $z$. That is, \[\Phi(f(x,y,z))=f(z,y,x),\quad \text{for all $f\in\mathbb{F}$.}\] To understand the meaning of this automorphism, notice that \[S':=\{0,z,z+y,z+y+x\}=-S+x+y+z\] is a reflection of $S$. The vertices of $T$ obey the following inequalities:
\begin{gather*} 
f_{010}<f_{01}<0<1<f_{00}<f_0,\quad f_{01}<f_{011}<1, \\
f_{110}<f_0<f_{11}<1<f_{111}<f_1, \quad 0<f_{10}<1.
\end{gather*}

\begin{table}
\begin{align*}
    f_0 &= \frac{x+y+z}{x}, & f_1 &= \frac{x+y+z}{x+y}, \\
    f_{00}&=\frac{x+y}{x}, & f_{10}&=\frac{x}{x+y},\\
    f_{01}&=\frac{-xy-y^2-yz}{xz}, & f_{11} &=\frac{yz+xy+y^2}{xy+y^2+xz+yz},\\
    f_{010}&=\frac{-x^2-2xy-y^2-yz-xz}{xz}, & f_{110}&=\frac{-x^2-xy-xz}{xy+y^2+xz+yz},\\
    f_{011}&=\frac{-y^2+xz-yz}{xz}, & f_{111}&=\frac{2yz+xy+y^2+xz}{xy+y^2+xz+yz}.
\end{align*}
    \caption{The rational functions from the tree in Figure~\ref{fig:size4}.}
    \label{tab:polys}
\end{table}

If $f_{00}=f_{1}$, then setting $k=f_{00}=f_{1}$ we have that $S=\{0,x,kx,k^2 x\}$ is symmetric of type $(2,4)$. Next, from the equation
$$xzf_{011}=(xy+y^2+xz+yz)(f_{10}-f_{11})$$
we have $f_{011}=0$ if and only if $f_{10}=f_{11}$. But $f_{011}=0$ implies that $\Phi(f_{00})=\Phi(f_1)$. Moreover, by the first sentence of this paragraph, this implies that $S'$ is symmetric of type $(2,4)$ and $S$ is symmetric of type $(1,3)$. We conclude that if $f_{00}=f_1$ or $f_{011}=0$ or $f_{10}=f_{11}$, then $S$ is symmetric. In these cases Maker has a 4-move strategy (Theorem \ref{4}).

Among the relations $g=p-q$ left to consider, where $p$ and $q$ are distinct (as elements of $\mathbb{F}$) vertices of $T$, most have no positive solutions to $g=0$ since the coefficients of the numerator of the rational function $g$ have the same sign. There are three relations that cannot be ruled out by this reasoning:
\begin{equation*}
    g_1 = f_{11}-f_{011},\quad g_2 = f_{110}-f_{01}, \quad g_3 = f_{110}-f_{011}.
\end{equation*}
The significance of $g_1,g_2,$ and $g_3$ is that they represent exactly those pairs of vertices that could be equal for some choice of positive $(x,y,z)\in\mathbb{Q}^3$.

Now, observe that if $T$ gives a $5$-move strategy for $S'$, then the tree $\Phi(T)$ obtained by applying $\Phi$ to every vertex of $T$ gives a $5$-move strategy for $S$ (by symmetry, this remains true after swapping $S$ and $S'$). Therefore, Maker has a 5-move strategy if the tree provides one for $S$ or $S'$. The tree fails for both $S$ and $S'$ only if $(x,y,z)$ vanishes on $g_i$ and on $\Phi(g_j)$ for some pair of indices $1\leq i\leq j\leq 3$. For each pair $(i,j)$ we look for rational points on the intersection of the curves $g_i=0$ and $\Phi(g_j)=0$, hoping not to find any.

\begin{table}
\centering
\renewcommand{\arraystretch}{1.15}
\begin{tabular}{c|c}
$(i,j)$ & $p(x)$\\
\hline
$(1,1)$ & $x^5-4x^3-6x^2-4x-1$\\
$(1,3)$ & $x^7-4x^6-12x^5-4x^4+10x^3+11x^2+5x+1$\\
$(2,3)$ & $x^8 + x^7 - 7x^6 - 15x^5 - 3x^4 + 19x^3 + 23x^2 + 11x + 2$\\
$(3,3)$ & $3x^8+17x^7+23x^6-6x^5-34x^4-22x^3+2x^2+7x+2$
\end{tabular}
\caption{Some polynomials used in the proof of Theorem \ref{5}.}
    \label{tab:polys2}
\end{table}

First consider the case $(i,j) = (2,2)$. This corresponds to the system $g_2=\Phi(g_2)=0$. Notice that these functions are homogeneous: if $f=g_2$ or $f=\Phi(g_2)$ then $f(x,y,z)=f(\lambda x,\lambda y,\lambda z)$ for any constant $\lambda \neq 0$. In particular, if $(x,y,z)$ is a root then so is any multiple $\lambda(x,y,z)$. Set $\lambda=1/y$, or equivalently set $y=1$. Clearing denominators, we get the system of equations $h=\Phi(h)=0$ in $x$ and $z$, where
$$h = -x^3z-x^2z^2+xz^2+x^2+3xz+z^2+2x+2z+1.$$
Now we notice that $xz(1+x+z)(x-z)=\Phi(h)-h=0$. Then since $(x,z)$ lies in the first quadrant of the $xz$-plane, we must have $x-z=0$. Setting $z$ equal to $x$ in $h$, we get
$$-2x^4 + x^3 + 5x^2 + 4x + 1=0.$$
One checks that $\frac{1}{2}$ is not a root of this polynomial. By the rational root theorem, the polynomial has no positive rational roots.

Next consider $(i,j)=(1,2)$, corresponding to the system $g_1=\Phi(g_2)=0$. As before, we set $y=1$ and clear denominators to get the system
\begin{align*}
-x^2z^2+xz^2+2xz+z^2+x+2z+1&=0, \\
-x^2z^2-xz^3+x^2z+x^2+3xz+z^2+2x+2z+1&=0.
\end{align*}
By a Gr\"obner basis computation (refer to \cite{CLO15} more information about Gr\"obner bases) in the ring $(\mathbb{Q}[x])[z]$ we obtain the equivalent system of equations
\begin{align*}
x^7-x^6-5x^5-3x^4&=0, \\
(2x^3+2x^2)z-x^6+5x^4+6x^3+2x^2&=0,\\
z^2 +(-2x^2+2)z+2x^6-2x^5-9x^4-7x^3-3x^2+1&=0.
\end{align*}
From the first equation, if $x> 0$ then $x$ must be a root of the polynomial
$x^3-x^2-5x-3.$ Using the rational root theorem, we find that the only positive rational root of this polynomial is $x=3$, leading to the solution $(x,y,z)=(3,1,2)$. In this case $S'$ is a copy of $\{0,2,3,6\}$. Then Example \ref{0236} shows that Maker has a $5$-move strategy for $S'$, and hence for $S$.

The series of computations for $(i,j)=(1,2)$ can be repeated for all of the remaining tuples. Table \ref{tab:polys2} records for each tuple the univariate polynomial $p$ obtained with the property that $p(x)=0$ whenever a point $(x,1,z)$ in the positive octant solves $g_i=\Phi(g_j)=0$. None of these polynomials have positive rational roots.
\end{proof}

Below is the Macaulay2 \cite{Macaulay2} code used to carry out the computations at the end of the proof of Theorem \ref{5}. The code provided is for the tuple $(i,j)=(1,2)$, but can be easily modified for the other cases.

\begin{verbatim}
F = frac(QQ[x,y,z]);
f01 = (-x*y-y^2-y*z)/(x*z);
f11 = (y*z+x*y+y^2)/(x*y+y^2+x*z+y*z);
f011 = (-y^2+x*z-y*z)/(x*z);
f110 = (-x^2-x*y-x*z)/(x*y+y^2+x*z+y*z);
g = {f11-f011, f110-f01, f110-f011};

R = frac(QQ[x,z]);
dehom = map(R,F,{x,1,z});
Phi = map(R,R,{z,x});

i = 1; j = 2; -- modify code here

p = dehom(g_(i-1)); q = Phi(dehom(g_(j-1)));
d = denominator(p); p = p*d; q = q*d;
I = ideal{p,q}; S = (QQ[x])[z]; I = sub(I,S);

gens gb I
\end{verbatim}

\vskip20pt\noindent {\bf Acknowledgements.} We thank Joe Briggs for suggesting this problem in the Auburn Graduate Student Research Seminar and the other graduate students in the research seminar, particularly Haile Gilroy and Evan Leonard. We thank Dömötör Pálvölgyi for bringing the application of the Hales-Jewett theorem to our attention.

\end{document}